\documentclass[reqno]{amsart}
\usepackage{amsmath}
\usepackage{amssymb}
\usepackage{amsthm}
\usepackage{mathrsfs}
\usepackage{ulem}
\usepackage{url}
\usepackage{cancel}
\newtheorem{definition}{Definition}[section]
\newtheorem{theorem}[definition]{Theorem}
\newtheorem{lemma}[definition]{Lemma}

\newtheorem{cor}[definition]{Corollary}

\newtheorem{example}[definition]{Example}
\newtheorem{examples}[definition]{Examples}
\newtheorem{remark}[definition]{Remark}

\usepackage{color}

\newcommand{\M}{\mathrm{Mult}(H_k)}
\newcommand{\C}{\mathbb{C}}
\newcommand{\D}{\mathbb{D}}
\newcommand{\B}{\mathbb{B}}
\newcommand{\HH}{H^{\infty}(\D)}
\newcommand{\bb}{\hat{b}(\lambda)}
\DeclareRobustCommand{\erase}{\bgroup\markoverwith{\textcolor{red}{\rule[.5ex]{2pt}{0.4pt}}}\ULon}

\begin{document}

\title[Complex structure that admits CNP spaces of Hardy type]{Complex structure that admits complete Nevanlinna--Pick spaces of Hardy type}
\author{Kenta Kojin}
\address{
Graduate School of Mathematics, Nagoya University, 
Furocho, Chikusaku, Nagoya, 464-8602, Japan
}
\email{m20016y@math.nagoya-u.ac.jp}
\date{\today}

\begin{abstract}
    In this paper, we will characterize those sets, over which every irreducible complete Nevanlinna--Pick space enjoys that its multiplier and supremum norms coincide. Moreover, we will prove that, if there exists an irreducible complete Nevanlinna--Pick space of holomorphic functions on a reduced complex space $X$ whose multiplier algebra is isometrically equal to the algebra of bounded holomorphic functions (we will say that such a space is {\bf of Hardy type} in this paper), then $X$ must be biholomorphic to the unit disk minus a zero analytic capacity set. This means that the Hardy space is characterized as a unique irreducible complete Nevanlinna--Pick space of Hardy type.
\end{abstract}

     \maketitle

    
\section{Introduction}
    
    In 1967, Sarason \cite{Sar} gave an operator theoretic proof of the celebrated interpolation theorem due to Pick \cite{Pic} and Nevanlinna \cite{Nev1919}. 
    He identified the bounded holomorphic functions $\HH$ on the open unit disk $\mathbb{D}$ with the multiplier algebra of the Hardy space $H^2$ on $\mathbb{D}$, a reproducing kernel Hilbert space (RKHS for short) of holomorphic functions on $\D$, and rephrased the Nevanlinna--Pick interpolation theorem in terms of reproducing kernel Hilbert spaces.
    
    In \cite{Agl88}, Agler searched for another RKHS on which a Nevanlinna--Pick type interpolation theorem holds. The existing works due to Agler \cite{Agl88}, McCullough \cite{McC1992, McC1994} and Quiggin \cite{Qui} altogether established a characterization theorem for RKHS's on which the matrix-valued Pick interpolation theorem holds. Such a space is called a complete Nevanlinna--Pick space. The Hardy space $H^2$ and the Dirichlet space on $\D$ are prototypical examples (see \cite{AMp}). 
    Then, Agler and McCarthy \cite{AM2000} established the following ``universality" of the Drury--Arveson space $H_d^2$ (which is understood as a multivariable Hardy space on the $d$-dimensional open unit ball $\B_d$): Every irreducible complete Nevanlinna--Pick space can be embedded into $H_d^2$ with some $1\le d\le\infty$. Then, irreducible complete Nevanlinna--Pick spaces have been studied by many hands, because those spaces share many properties with $H^2$ by replacing $\HH$ with their multiplier algebras. For example, we can generalize Beurling's theorem and the inner-outer factorization for those spaces (see \cite{AMp, JM}). See \cite{Hartz2022} as a good survey including recent developments, and \cite{AMp} as a standard reference.
    
    It is well known that the multiplier algebra of $H^2$ is isometrically equal to $\HH$. Moreover, for $1\le d<\infty$ and $d<v$, the multiplier algebras of the Hardy space $H^2(\B_d)$ on $\mathbb{B}_d$ and the weighted Bergman space $L_a^2((1-|z|^2)^{v-d-1}dV)$ on $\mathbb{B}_d$ are isometrically equal to $H^{\infty}(\B_d)$, but it is a folklore that those spaces are not complete Nevanlinna--Pick spaces. We would also like to point out that the multiplier norm of an (even irreducible) complete Nevanlinna--Pick space does not have to be comparable with the supremum norm. Thus, the following subtle questions naturally arise: 
    \begin{itemize}
    \item When does the multiplier norm of an irreducible complete Nevanlinna--Pick space coincide with the supremum norm? 
    \item When is the multiplier algebra of an irreducible complete Nevanlinna--Pick space of holomorphic functions on a ``complex space" $X$ isometrically equal to all the bounded holomorphic functions on $X$? 
    \end{itemize}
    The purpose of this paper is to give complete answers to the questions.
    
    \medskip
    We first study irreducible complete Nevanlinna--Pick spaces $H_k$ with the following norm property: $\|\phi\|_{\M}=\|\phi\|_{\infty}$ for any $\phi\in\M$. Here and in what follows, $H_k$ is an RKHS with reproducing kernel $k$, and $\M$ denotes its multiplier algebra. A consequence of our investigation is that one can choose $d=1$ in the above-mentioned universality of $H_d^2$ (see Theorem \ref{theorem1}). This is indeed an answer to the first question. 
    
    \medskip
    We are then interested in irreducible complete Nevanlinna--Pick spaces $H_k$ consisting of holomorphic functions on a reduced complex space $X$ such that $\M$ is isometrically equal to $H^{\infty}(X)$. Here, a reduced complex space is a complex manifold allowing singularities (see Definition \ref{def:complex space}).  We say that such an RKHS is of Hardy type because a prototypical example is the Hardy space $H^2$ as mentioned before.
    
    By Riemann removable singularities theorem \cite[Theorem 10.20]{Rud}, we can construct an RKHS of Hardy type on $\D\setminus\{a\}$ for any $a\in\D$. 
    This observation suggests that the characterization problem for RKHS of Hardy type may be related to the removable singularities problem for bounded holomorphic functions on a subdomain of $\mathbb{D}$.
    Riemann's theorem was strengthened by means of the concept of analytic capacity (see e.g., \cite{Gar, Tol}).
    Indeed, we will also prove the following satisfactory result involving the analytic capacity (see Theorem \ref{theorem2}):
    \begin{theorem}
    Let $X$ be a reduced complex space with arbitrary $x_0\in X$. Then, there exists an RKHS $H_k$ of Hardy type on $X$ normalized at $x_0$ if and only if there exist a relatively closed set $E\subset\D$ and a 
    biholomorphic map $j$ from X onto $\D\setminus E$ such that 
    \begin{itemize}
    \item[(1)] $\D\setminus E$ is connected, 
    \item[(2)] the analytic capacity of $E$ is $0$,
    \item[(3)] $j$ sends $x_0$ to the origin $0 \in \mathbb{D}$.
    \end{itemize}
    Moreover, in this case, such an RKHS $H_k$ is unique and its reproducing kernel is given by
    \begin{equation*}
        k(x,y)=\frac{1}{1-j(x)\overline{j(y)}}\;\;\;\;(x,y\in X).
    \end{equation*}
    Hence, 
    \begin{equation*}
        U: H^2\rightarrow H_k, \;\;\;\; f\mapsto f\circ j
    \end{equation*}
    is a unitary operator. 
    \end{theorem}
    
    Related to this result, it has been known, see \cite[Proposition 8.83]{AMp} and \cite[Corollary 3.3]{Har2017}, that there is no RKHS of Hardy type on both the polydisk $\D^d$ and $\B_d$ with $d\ge2$. Therefore, the above theorem may be regarded as a vast generalization of this observation (see also Corollary \ref{corollary1}).
    
    We will also prove the following theorem (see Theorem \ref{prop1}):
    \begin{theorem}
    Let $H_k$ be an RKHS of Hardy type on a reduced complex space $X$.
    If $X$ is complete with respect to the Carath\'{e}odory distance, then $X$ must be biholomorphic to $\D$.
    \end{theorem} 

    
    These two theorems are regarded as perfect answers to the second question given above. 
    
    \medskip
    Our strategy to prove the above-mentioned results is the use of the following naive idea: The geometry of $X$ can be captured from that of an RKHS on $X$.
    
    In our discussion, we will crucially use a solution of the extremal problem for the following distance:
    \begin{equation*}
        d_k(x,y):=\sqrt{1-\frac{|k(x,y)|^2}{k(x,x)k(y,y)}}\;\;\;\;(x,y\in X).
    \end{equation*}
    See e.g., \cite[Proposition 3.1]{Har2017}. When an irreducible RKHS $H_k$ on $X$ has the two-point Nevanlinna--Pick property (see Definition \ref{def:two-point Pick}), the solution we will use is:   
    \begin{equation*} 
        \sup\{\mathrm{Re}\phi(x)\;|\;\|\phi\|_{\M}\le 1 \;\mathrm{and}\; \phi(x_0)=0\}=d_k(x,x_0)
    \end{equation*}
    must hold for any $x, x_0\in X$, and moreover, for any $x\in X$ with $x\ne x_0$, there is a unique multiplier $\phi_{x}$ that achieves the supremum of the left-hand side. This fact plays a crucial role to answer the first question and to prove the first theorem above. 
    
    In addition, if $H_k$ consists of holomorphic functions on a reduced complex space $X$ and $\M=H^{\infty}(X)$ isometrically, then the left-hand side of the above identity coincides with the M$\ddot{\mathrm{o}}$bius pseudo-distance
    \begin{equation*}
    d_X(x,x_0):=\sup_{f\in H(X,\D)}d_{\D}(f(x), f(x_0))
    \end{equation*}
    (see Lemma \ref{lemma4-1}).
    Here $\displaystyle d_{\D}(z,w):=\left|\frac{z-w}{1-\overline{w}z}\right|$ and $H(X,\D)$ is the set of holomorphic functions from $X$ into $\D$. 
    Namely, the complex geometric structure of $X$ is certainly connected to the geometry of $H_k$.
    Actually, this observation motivated us to solve the second question mentioned above.


\section{Complete Nevanlinna--Pick spaces}
 Let $H_k$ be an RKHS on a set X with reproducing kernel $k:X\times X\rightarrow\C$. We will always assume that $X$ is neither the empty set nor a singleton (in the latter case, $H_k=\C$), and  $H_k$ is irreducible in the following sense: For any two points $x,y\in X$, the reproducing kernel $k$ satisfies $k(x,y)\ne 0$, and if $x\ne y$, then $k(\cdot, x)$ and $k(\cdot, y)$ are linearly independent in $H_k$. A function $\phi:X\rightarrow\C$ is called a multiplier of $H_k$ if $\phi f\in H_k$ for all $f\in H_k$. In this case, the multiplication operator $M_{\phi}:H_k\rightarrow H_k$ defined by $M_{\phi}f:=\phi f$ is bounded by the closed graph theorem, and the multiplier norm $\|\phi\|_{\M}$ is defined to be the operator norm of $M_{\phi}$. The multiplier algebra of $H_k$, denoted by $\M$, is the algebra consisting of multipliers of $H_k$. It is naturally identified with a subalgebra of $B(H_k)$. 

The following normalization assumption is useful when one studies multiplier algebras: An RKHS $H_k$ is normalized at $x_0\in X$ if $k(x, x_0)=1$ for all $x\in X$. 
Any non-vanishing kernel can be normalized by considering
 \begin{equation}\label{equation:normalization}
    k_{x_0}(x,y):=\frac{k(x,y)k(x_0,x_0)}{k(x,x_0)k(x_0,y)}
 \end{equation}
 for any fixed $x_0\in X$. This normalization does not change the structure of the multiplier algebra (see \cite[Section 2.6]{AMp}). Therefore, we will assume that every irreducible RKHS that we consider below is normalized at some point.

 In this paper, we are interested in complete Nevanlinna--Pick spaces that enjoy $\|\phi\|_{\M}=\|\phi\|_{\infty}$ for all $\phi\in\M$, where $\|\phi\|_{\infty}:=\sup\{\phi(x)\;|\; x\in X\}$. For instance, if $H_k$ 
 is the Hardy space on the open unit disk $\D$, then all multipliers satisfy this norm condition. Moreover, we can prove that Mult$(H^2)$$=\HH$ isometrically \cite[Theorem 3.24]{AMp}.

 We recall the definition of complete Nevanlinna--Pick spaces. One of the generalizations of the Hardy space $H^2$ is the Drury--Arveson space $H^2_d$, whose reproducing kernel is $\displaystyle k_d(z,w)=\frac{1}{1-\langle z,w\rangle_{\C^d}}$ $(z,w\in \B_d)$. Here, $1\le d\le \infty$ and $\B_d:=\{z\in\C^d\;|\;\|z\|_{\C^d}<1\}$. 

 \begin{definition}\label{def:CNP}
An RKHS $H_k$ on a set $X$ is said to be a {\bf complete Nevanlinna--Pick space} if whenever $m, n\in \mathbb{N}$, $x_1\ldots, x_n$ are points in $X$ and $W_1,\ldots, W_n$ are $m\times m$ matrices such that
\begin{equation*}
    [(I-W_iW_j^*)k(x_i, x_j)]_{i,j=1}^n
\end{equation*}
is positive semidefinite, then there exists a $\Phi\in\mathbb{M}_m(\C)\otimes\M$ with $\Phi(x_i)=W_i$ for $1\le i\le n$ and $\|\Phi\|_{\mathbb{M}_m(\C)\otimes\M}\le 1$.
 \end{definition}

Agler and McCarthy \cite{AM2000} established the universality of the Drury--Arveson space. 

\begin{theorem}[Agler--McCarthy $\mbox{\cite{AM2000}}$]\label{universality}
Let $H_k$ be an irreducible complete Nevanlinna--Pick space on a set $X$ normalized at $x_0\in X$. Then, there exist a number $1\le d\le \infty$ and an injection $b:X\rightarrow\B_d$ with $b(x_0)=0$ such that
\begin{equation*}
    k(x,y)=\frac{1}{1-\langle b(x), b(y)\rangle_{\C^d}}\;\;\;\;(x,y\in X).
\end{equation*}
\end{theorem}

\begin{remark}\label{remarkHartz}
\upshape In this case, Hartz \cite[Theorem 2.1]{Har2015} proved that the composition operator $C_b:H_d^2\ominus I(A)\rightarrow H_k$ defined by $C_bf:=(f|_A)\circ b$ is a unitary operator, where $A:=b(X)$ and $I(A):=\{f\in H_d^2\;|\;f(z)=0\;(z\in A)\}$.
\end{remark}



\section{Uniqueness of the Hardy space $H^2$}

In this section, we will give a characterization of normalized irreducible complete Nevanlinna--Pick spaces such that $\|\phi\|_{\M}=\|\phi\|_{\infty}$ holds for all $\phi\in\M$. 

\begin{definition}
An RKHS $H_k$ of functions on a set $X$ is said to be of {\bf weak Hardy type} if $H_k$ is a normalized irreducible complete Nevanlinna--Pick space such that every $\phi\in\M$ satisfies 
\begin{equation*}
\|\phi\|_{\M}=\|\phi\|_{\infty}.
\end{equation*}
\end{definition}

\begin{remark}
\upshape If the diagonal $k(x,x)$ of a reproducing kernel is bounded on $X$, then \cite[Propositions 4.3 and 4.7]{Ser} imply that $H_k$ never becomes of weak Hardy type. In particular, if $X$ is a compact set, then there is no RKHS of weak Hardy type with the continuous diagonal $k(x,x)$. 
\end{remark}

\begin{definition}
Let $A$ be a subset of the open unit disk $\D$. The $A$ is said to be a {\bf set of uniqueness for $H^2$} if the only element of $H^2$ that vanishes on $A$ is the zero function. The $A$ is said to be {\bf dominating} if
\begin{equation*}
\sup\{|\Phi(z)|\;|\;z\in A\}=\sup\{|\Phi(z)|\;|\;z\in \D\}
\end{equation*}
holds for all $\Phi\in\HH$; see \cite[Definition 4.10]{RS}.
\end{definition}

\begin{remark}\label{remark:uniqueness}
\upshape  (1) Since the zeros of a nonzero function in $H^2$ satisfy the Blaschke condition \cite[Theorem 15.23]{Rud}, $A\subset\D$ is a set of uniqueness for $H^2$ if and only if
\begin{equation*}
\sum_{a\in A}(1-|a|)=\infty.
\end{equation*}
(2) We can see that, if $A\subset\D$ is dominating, then it is a set of uniqueness for $H^2$. In fact, if $\displaystyle\sum_{a\in A} (1-|a|)<\infty$, then
\begin{equation*}
\Theta(z):=\prod_{a\in A}\frac{a-z}{1-\overline{a}z}\frac{|a|}{a}\;\;\;\;(z\in\D)
\end{equation*}
converges and defines an inner function by \cite[Theorem 15.21]{Rud}. However, $\Theta$ has no zeros except the points $a\in A$. This contradicts that $A$ is dominating.

\end{remark}

We will use the following lemma to prove the main theorem of this section:


\begin{lemma}\label{lemma1}
Let $X$ be a set and $x_0\in X$. Let $j$ be an injective function from $X$ into $\D$ such that $j(x_0)=0$. Define a normalized irreducible complete Nevanlinna--Pick kernel $k$ on $X\times X$ by
\begin{equation*}
k(x,y)=\frac{1}{1-j(x)\overline{j(y)}}.
\end{equation*}
Then, for any $\phi\in\M$, there exists a $\Phi\in H^{\infty}(\D)$ such that $\phi=\Phi\circ j$ and $\|\phi\|_{\M}=\|\Phi\|_{\infty}$.
\end{lemma}
\begin{proof}
This immediately follows from \cite[Theorems 6.5 and 8.30]{AMp} (see also \cite[Theorem 4.19]{Hartz2022}).
\end{proof}


\begin{definition}\label{def:two-point Pick}
We say that an RKHS $H_k$ on a set $X$ satisfies the  {\bf two-point Nevanlinna--Pick property} if whenever $x_1, x_2$ are points in $X$ and $w_1,w_2$ are complex numbers such that
\begin{equation*}
\begin{bmatrix}
    (1-|w_1|^2)k(x_1, x_1)&(1-w_1\overline{w_2})k(x_1, x_2)\\
    (1-w_2\overline{w_1})k(x_2, x_1)&(1-|w_2|^2)k(x_2, x_2)
\end{bmatrix}
\end{equation*}
is positive semidefinite, then there exists a $\phi\in\M$ with $\phi(x_i)=w_i$ for $i=1,2$ and $\|\phi\|_{\M}\le 1$. This matrix is nothing but the $m=1$, $n=2$ case of the matrix in Definition \ref{def:CNP}.
\end{definition}

Obviously, every complete Nevanlinna--Pick space satisfies this property. 

In this paper, a solution of the extremal problem for the following 
pseudo-distance on a set $X$ derived from an RKHS $H_k$ is important:

\begin{equation*}
    d_k(x,y):=\sqrt{1-\frac{|k(x,y)|^2}{k(x,x)k(y,y)}}\;\;\;\;(x,y\in X).
\end{equation*}

Here is the precise statement of the extremal theorem.

\begin{lemma}[\mbox{\cite[Proposition 3.1]{Har2017}}]\label{lemma2}
Let $H_k$ be an irreducible RKHS on a set $X$ normalized at $x_0\in X$. Suppose $H_k$ satisfies the two-point Nevanlinna--Pick property. Then
\begin{equation*} 
    \sup\{\mathrm{Re}\phi(x)\;|\;\|\phi\|_{\M}\le 1 \;\mathrm{and}\; \phi(x_0)=0\}=d_k(x,x_0)
    =\left(1-\frac{1}{k(x,x)}\right)^{1/2}
\end{equation*}
for any $x\in X$, and this number is strictly positive if $x\ne x_0$. Moreover, there is a unique multiplier $\phi_{x}$ that achieves the supremum if $x\ne x_0$.
\end{lemma}

\begin{remark}\label{remark:invariant}
\upshape Let $H_k$ be an irreducible RKHS on a set $X$. For each point $x_0\in X$, we can construct a normalized reproducing kernel $k_{x_0}$ by equation (\ref{equation:normalization}). Then, it is easy to check that 
\begin{equation*}
d_k(x,y)=d_{k_{x_0}}(x,y)
\end{equation*}
holds for every $x,y\in X$. 
\end{remark}



Here is the main theorem of this section.

\begin{theorem}\label{theorem1}
An RKHS $H_k$ on a set $X$ normalized at $x_0\in X$ is of weak Hardy type if and only if there exist a dominating set $A\subset\D$ and a bijection $j:X\rightarrow A$ such that $j\in\M$, $j(x_0)=0$ and
\begin{equation}\label{equation1}
k(x,y)=\frac{1}{1-j(x)\overline{j(y)}}\;\;\;\;(x,y\in X).
\end{equation}
Hence, 
\begin{equation*}
U: H^2\rightarrow H_k, \;\;\;\; f\mapsto f\circ j
\end{equation*}
is a unitary operator.
\end{theorem}


\begin{proof}
If $H_k$ is an irreducible complete Nevanlinna--Pick space normalized at $x_0$, Theorem \ref{universality} shows that there exist a number $1\le d\le \infty$ and an injection $b:X\rightarrow\B_d$ with $b(x_0)=0$ such that
\begin{equation*}
    k(x,y)=\frac{1}{1-\langle b(x), b(y)\rangle_{\C^d}}\;\;\;\;(x,y\in X).
\end{equation*}
First, we will prove that $\langle b(\cdot), \xi\rangle_{\C^d}$ is a contractive multiplier of $H_k$ for any $\xi\in \C^d$ with $\|\xi\|_{\C^d}= 1$. Let $\{e_{\alpha}\}_{\alpha\in A}$ be an orthonormal basis for $\C^d$ containing $\xi$. By Parseval's identity (see e.g., \cite[Theorem 4.18]{Rud}), we have
\begin{align*}
   &(1-\langle b(x),\xi\rangle\overline{\langle b(y),\xi\rangle})k(x,y)\\
   &=\frac{1-\langle b(x), \xi\rangle \langle \xi, b(y)\rangle}{1-\langle b(x), b(y)\rangle}\\
   &= \frac{1-\displaystyle\sum_ {\alpha\in A}\langle b(x), e_{\alpha}\rangle \langle e_{\alpha}, b(y)\rangle
   +\displaystyle\sum_{\alpha\in A:e_{\alpha}\ne\xi}\langle b(x), e_{\alpha}\rangle \langle e_{\alpha}, b(y)\rangle}
   {1-\langle b(x), b(y)\rangle}\\
   &=\frac{1-\langle b(x), b(y)\rangle+\displaystyle\sum_{\alpha\in A:e_{\alpha}\ne \xi}\langle b(x), e_{\alpha}\rangle \langle e_{\alpha}, b(y)\rangle}{1-\langle b(x), b(y)\rangle}\\
   &=1+\displaystyle\sum_{\alpha\in A:e_{\alpha}\ne\xi}\langle b(x), e_{\alpha}\rangle \langle e_{\alpha}, b(y)\rangle k(x,y)\;\;\;\;(x,y\in X).
\end{align*}
Therefore, $\displaystyle (1-\langle b(x),\xi\rangle\overline{\langle b(y),\xi\rangle})k(x,y)$ is positive semidefinite by the Schur product theorem \cite[Theorem A.1]{AMp}. Hence, \cite[Corollary 2.37]{AMp} tells that $\langle b(\cdot), \xi \rangle$ is a contractive multiplier of $H_k$.

We fix a point $\lambda\in X\setminus \{x_0\}$ and set $\displaystyle \bb:=\frac{b(\lambda)}{\|b(\lambda)\|}$ (n.b., $b(\lambda)\ne 0$ because $b$ is injective and $b(x_0)=0$). Then, $\langle b(\cdot), \bb\rangle$ is a contractive multiplier of $H_k$ such that $\langle b(x_0), \bb\rangle=0$ and 
\begin{equation*}
\langle b(\lambda),\bb\rangle=\|b(\lambda)\|=\left(1-\frac{1}{k(\lambda, \lambda)}\right)^{1/2}. 
\end{equation*}
Let $\{e_{\alpha}\}_{\alpha\in A}$ be an orthonormal basis for $\C^d$ containing $\bb$. For any $\alpha\in A$ with $e_{\alpha}\ne \bb$, we define a multiplier of $H_k$ by
\begin{equation*}
\langle b(\cdot),\bb\rangle+\frac{1}{2}\langle b(\cdot), e_{\alpha}\rangle^2.
\end{equation*}
Obviously, $\langle b(x_0), \bb\rangle+\frac{1}{2}\langle b(x_0),e_{\alpha}\rangle^2=0$. Moreover, we have
\begin{equation*}
\langle b(\lambda),\bb\rangle+\frac{1}{2}\langle b(\lambda), e_{\alpha}\rangle^2=\langle b(\lambda), \bb\rangle=\|b(\lambda)\|=\displaystyle\left(1-\frac{1}{k(\lambda,\lambda)}\right)^{1/2}
\end{equation*}
because $\bb$ and $e_{\alpha}$ are orthogonal.
We will prove that $\langle b(\cdot),\bb\rangle+\frac{1}{2}\langle b(\cdot),e_{\alpha}\rangle^2$ is contractive.
Since $|\langle b(x), \bb\rangle|^2+|\langle b(x), e_{\alpha}\rangle|^2\le \|b(x)\|^2<1$, 
we have
\begin{align*}
&|\langle b(x), \bb\rangle+\frac{1}{2}\langle b(x), e_{\alpha}\rangle^2|^2\\
&\le (|\langle b(x),\bb\rangle|+\frac{1}{2}|\langle b(x), e_{\alpha}\rangle|^2)^2\\
&= |\langle b(x), \bb\rangle|^2+|\langle b(x), \bb\rangle||\langle b(x), e_{\alpha}\rangle|^2+\frac{1}{4}|\langle b(x), e_{\alpha}\rangle|^4\\
&\le |\langle b(x),\bb\rangle|^2+|\langle b(x), \bb\rangle|(1-|\langle b(x),\bb\rangle|^2)+\frac{1}{4}(1-|\langle b(x),\bb\rangle|^2)^2\\
&=\frac{1}{4}|\langle b(x), \bb\rangle|^4-|\langle b(x),\bb\rangle|^3+\frac{1}{2}|\langle b(x),\bb\rangle|^2+|\langle b(x), \bb\rangle|+\frac{1}{4}
\end{align*}
for all $x\in X$. A tedious calculation enables us to prove that a function 
\begin{equation*}
g(t):=\frac{1}{4}t^4-t^3+\frac{1}{2}t^2+t+\frac{1}{4}
\end{equation*}
is monotonically increasing on the interval $[0,1]$ with $g(1)=1$. Thus, we obtain
\begin{equation*}
\sup_{t\in [0,1]}g(t)=1.
\end{equation*}
Hence, we have 
\begin{equation*}
\left\|\langle b(\cdot), \bb\rangle+\frac{1}{2}\langle b(\cdot), e_{\alpha}\rangle^2\right\|_{\infty}\le 1
\end{equation*}
because $\|\langle b(\cdot),\bb\rangle\|_{\infty}$ is less than or equal to $1$.
Since $H_k$ is of weak Hardy type, we have
\begin{equation*}
\left\|\langle b(\cdot), \bb\rangle+\frac{1}{2}\langle b(\cdot), e_{\alpha}\rangle^2\right\|_{\M}=
\left\|\langle b(\cdot), \bb\rangle+\frac{1}{2}\langle b(\cdot), e_{\alpha}\rangle^2\right\|_{\infty}\le 1.
\end{equation*}
Therefore, the uniqueness part of Lemma \ref{lemma2} implies 
\begin{equation*}
\langle b(\cdot),\bb\rangle+\frac{1}{2}\langle b(\cdot), e_{\alpha}\rangle^2=\langle b(\cdot),\bb\rangle
\end{equation*}
for every $\alpha\in A$ with $e_{\alpha}\ne \bb$. So, we conclude that $\langle b(x), e_{\alpha}\rangle= 0$ for all $x\in X$ if $e_{\alpha}\ne \bb$. Thus, we have
\begin{equation}\label{equation2}
    b(x)=\sum_{\alpha\in A}\langle b(x), e_{\alpha}\rangle e_{\alpha}=\langle b(x), \bb\rangle \bb
\end{equation}
for all $x\in X$. We define $j:X\rightarrow \D$ by $j(x):=\langle b(x), \bb\rangle$. Since $b(x_0)=0$, we get $j(x_0)=0$.
The injectivity of $b$ and equation (3) imply that $j$ must be injective.
Moreover, using equation (\ref{equation2}) we can prove 
\begin{equation*}
    k(x,y)=\frac{1}{1-j(x)\overline{j(y)}}\;\;\;\;(x,y\in X).
\end{equation*}


 Next, we will show that $A:=j(X)$ is a set of uniqueness for $H^2$. On the contrary, suppose that $A$ is not. Then we have
\begin{equation*}
\sum_{x\in X}(1-|j(x)|)<\infty
\end{equation*}
(see Remark \ref{remark:uniqueness}). Therefore, \cite[Theorem 15.21]{Rud} tells us that
\begin{equation*}
\Theta(z):=\prod_{x\in X\setminus\{x_0\}}\frac{j(x)-z}{1-\overline{j(x)}z}\frac{|j(x)|}{j(x)}\;\;\;\; (z\in\D)
\end{equation*}
converges and defines an inner function. Moreover, $\Theta$ has no zeros except the points $j(x)$ ($x\in X\setminus\{x_0\}$). Since $\Theta$ is inner, \cite[Corollary 2.37 and Theorem 3.24]{AMp} show that 
\begin{equation*}
\frac{1-\Theta (j(x))\overline{\Theta (j(y))}}{1-j(x)\overline{j(y)}}\;\;\;\; (x,y\in X)
\end{equation*}
is positive semidefinite on $X\times X$. Hence by \cite[Corollary 2.37]{AMp} again, $\Theta\circ j$ is a multiplier of $H_k$. We define another multiplier $\theta$ of $H_k$ by
\begin{equation*}
\theta(x):=\frac{\Theta(j(x))}{\Theta(0)}\;\;\;\;(x\in X).
\end{equation*}
By the construction of $\Theta$, we have
\begin{align*}
\theta(x)&=
\begin{cases}
1 & (x=x_0)\\
0 & (x\ne x_0)
\end{cases}.
\end{align*}
Since $H_k$ is of weak Hardy type, $\theta$ is a contractive multiplier of $H_k$. However, the maximum modulus principle for multipliers \cite[Lemma 2.2]{AHMR} enable us to show that $\theta$ must be a constant function, a contradiction. Therefore, $A$ is a set of uniqueness for $H^2$. Hence, Remark \ref{remarkHartz} implies that
\begin{equation*}
U: H^2\rightarrow H_k, \;\;\;\; f\mapsto f\circ j
\end{equation*}
is a unitary operator.

It remains to prove that $A=j(X)$ is dominating. For any $\Psi\in H^{\infty}(\D)$, \cite[Corollary 2.37 and Theorem 3.24]{AMp} show that
\begin{equation*}
\frac{\|\Psi\|_{\infty}^2-\Psi(j(x))\overline{\Psi(j(y))}}{1-j(x)\overline{j(y)}}\;\;\;\;(x,y\in X)
\end{equation*}
is positive semidefinite on $X\times X$. Hence $\Psi\circ j$ is a multiplier of $H_k$ with 
\begin{equation*}
\|\Psi\circ j\|_{\M}\le \|\Psi\|_{\infty}
\end{equation*}
by \cite[Corollary 2.37]{AMp}. Thus, Lemma \ref{lemma1} yields a $\Phi\in\HH$ such that $\Psi\circ j=\Phi\circ j$ and $\|\Psi\circ j\|_{\M}=\|\Phi\|_{\infty}$. We note that 
\begin{equation*}
\sup\{|\Psi(z)|\;|\;z\in A\}=\|\Psi\circ j\|_{\M}=\|\Phi\|_{\infty}
\end{equation*}
because $H_k$ is of weak Hardy type.
Since $A$ is a set of uniqueness for $H^2$, we have $\Psi=\Phi$. Hence, we get
\begin{equation*}
\sup\{|\Psi(z)|\;|\;z\in A\}=\|\Phi\|_{\infty}=\|\Psi\|_{\infty}.
\end{equation*}
Therefore, we conclude that $A$ is dominating.


Conversely, with a reproducing kernel $k$ given by equation (\ref{equation1}), $H_k$ is a normalized irreducible complete Nevanlinna--Pick space. Since $A$ is dominating, 
Lemma \ref{lemma1} enable us to prove that $\|\phi\|_{\M}=\|\phi\|_{\infty}$ for any $\phi\in\M$. Therefore, $H_k$ is of weak Hardy type.
\end{proof}

Note that Hartz \cite{Har2015} proved that every irreducible complete Nevanlinna--Pick space whose multipliers are hyponormal is isomorphic to the Hardy space $H^2$ as an RKHS. 
As an application of Theorem \ref{theorem1}, we can see that whenever $H_k$ is of weak Hardy type, every multiplier is hyponormal.
However, the author does not know how to prove this fact without using Theorem \ref{theorem1}.


\section{Characterization of $X$ equipped with a complex structure}

In this section, we will consider normalized irreducible complete Nevanlinna--Pick spaces of holomorphic functions on a reduced complex space $X$ such that $\M$ is isometrically equal to
 $H^{\infty}(X)$. Here, $H^{\infty}(X)$ is the Banach algebra of all bounded holomorphic functions on $X$. We begin by recalling the definition of reduced complex spaces.

Let $U$ be an open set in $\C^n$. A subset $V\subset U$ is said to be an {\bf analytic subset of $U$} if for any $z\in U$ there exist an open neighborhood $U_z\subset\C^n$ of $z$ and holomorphic functions $f_1,\ldots, f_m$ on $U_z$ such that
\begin{equation*}
V\cap U_z=\{\zeta\in U_z\;|\;f_1(\zeta)=\cdots =f_m(\zeta)=0\}.
\end{equation*}
A function $f:V\rightarrow \C$ is holomorphic if for each $z\in V$ 
there exist an open neighborhood $U_z\subset\C^n$ of $z$ and a holomorphic function $F_z$ defined on $U_z$ such that $F_z(\zeta)=f(\zeta)$ for all $\zeta\in V\cap U_z$.
A mapping $F=(f_1,\ldots, f_m):V\rightarrow W$ between analytic subsets $V, W$ of open sets $U_V\subset\C^n$, $U_W\subset\C^m$ is holomorphic if each $f_i$ is holomorphic on $V$. A holomorphic mapping $F:V\rightarrow W$ is biholomorphic if it has a holomorphic inverse.

\begin{definition}\label{def:complex space}
Let $X$ be a second countable Hausdorff topological space. We say that $X$ is a {\bf reduced complex space} if there exist an open covering $\{X_{\alpha}\}_{\alpha\in A}$ of $X$ and homeomorphisms $\varphi_{\alpha}:X_{\alpha}\rightarrow V_{\alpha}$, where $V_{\alpha}$ is an analytic subset of $U_{\alpha}\subset\C^{n_{\alpha}}$, such that $\varphi_{\beta}\circ\varphi_{\alpha}^{-1}:\varphi_{\alpha}(X_{\alpha}\cap X_{\beta})\rightarrow\varphi_{\beta}(X_{\alpha}\cap X_{\beta})$ is biholomorphic for each non-empty intersection $X_{\alpha}\cap X_{\beta}$. 

We say that a function $f:X\rightarrow\C$ is holomorphic if for any chart $\varphi_{\alpha}:X_{\alpha}\rightarrow V_{\alpha}$ the function $f|_{X_{\alpha}}\circ\varphi_{\alpha}^{-1}:V_{\alpha}\rightarrow\C$ is holomorphic.
\end{definition}

\begin{examples}
\upshape (1) Every (second countable) complex manifold is a reduced complex space.

(2) Every analytic subset of an open set in $\C^n$ is a reduced complex space. In particular, Neil's parabola
\begin{equation*}
\{(z,w)\in \C^2\;|\;z^2-w^3=0\}
\end{equation*}
is an example of a reduced complex space but not a complex manifold. In fact, it has the origin $(0,0)$ as its singular point. Namely, there is no neighborhood $\Omega$ of $(0,0)$ such that $\Omega$ becomes a complex manifold.

\end{examples}

\begin{remark}
\upshape Reduced complex spaces are usually defined by means of structure sheaves (see e.g., \cite{GR, Gun}). We remark that the usual definition is equivalent to Definition \ref{def:complex space}; see \cite[page 163]{BM}.
\end{remark}

We recall the notion of the dimension of reduced complex spaces; see \cite[Chapter 5]{GR}.

\begin{definition}
Let $X$ be a reduced complex space. Then, for each $x\in X$ there exist an open neighborhood $\Omega$ and finitely many holomorphic functions $f_1, \ldots, f_m$ on $\Omega$ such that
\begin{equation}\label{equation:4}
\{\mathfrak{z}\in\Omega\;|\;f_1(\mathfrak{z})=\cdots =f_m(\mathfrak{z})=0\}=\{x\}.
\end{equation}
Among such systems $f_1, \ldots, f_m$ with equation $(\ref{equation:4})$, there exists one with minimal $m$. This minimal integer is called the {\bf dimension of $X$ at $x$} and denoted by $\dim_x{X}$. We say that $X$ is {\bf pure dimensional} if 
\begin{equation*}
\dim_x{X}=\dim_y{X}
\end{equation*}
holds for every $x,y\in X$.
\end{definition}

Obviously, $\dim_x{X}=0$ if and only if $x$ is an isolated point of $X$. It is not hard to see that Neil's parabola $\{(z,w)\in\C^2\;|\;z^2-w^3=0\}$ is pure $1$-dimensional. 

\begin{definition}\label{Hardytype}
An RKHS $H_k$ of holomorphic functions on a reduced complex space $X$ is of {\bf Hardy type} if $H_k$ is of weak Hardy type and $\M=H^{\infty}(X)$.
\end{definition}
Of course, the Hardy space $H^2$ is an example of such a space. 

If $D_a=\D\setminus\{a\}$ $(a\ne 0)$ and $H_{k_a}$ is the RKHS of holomorphic functions on $D_a$ given by reproducing kernel 
\begin{equation*}
k_a(z,w)=\frac{1}{1-z\overline{w}}\;\;\;\;(z,w\in D_a), 
\end{equation*}
then Riemann removable singularities theorem \cite[Theorem 10.20]{Rud} implies that $H_k$ is of Hardy type. This example suggests that a characterization problem of RKHS of Hardy type is related to the notion of analytic capacity because Riemann's theorem was strengthened by means of analytic capacity \cite{Gar, Tol}.

\begin{definition}
For a compact set $K\subset\C$, the number
\begin{equation*}
\gamma(K):=\sup\{|f'(\infty)|\;|\;f\in H^{\infty}(\C\setminus K) \;\mathrm{with}\; \|f\|_{\infty}\le 1\}
\end{equation*}
is called the {\bf analytic capacity} of $K$, where $\displaystyle f'(\infty):=\lim_{z\to \infty}z(f(z)-f(\infty))$. 
\end{definition}

For an arbitrary set $E\subset\C$, we define
\begin{equation*}
\gamma(E):=\sup_{K\subset E:\mathrm{compact}}\gamma(K).
\end{equation*} 
We will use the following generalization of Riemann removable singularities theorem to prove the main theorem in this section:

\begin{lemma}[\mbox{\cite[Theorem 1,4]{Gar}}]\label{capacity}
Let $E$ be a relatively closed subset of an open set $U\subset\C$ with $\gamma(E)=0$. Then, any $f\in H^{\infty}(U\setminus E)$ has an extension in $H^{\infty}(U)$.
\end{lemma}

The next lemma is a list of well-known facts on the analytic capacity. Since we do not know a suitable reference for these facts, a complete proof is provided.

\begin{lemma}\label{lemma:capacity}
Let $U\subset\C$ be a connected open set. If every $f\in H^{\infty}(U\setminus E)$ has an analytic extension to $U$, then we have the following:
\begin{enumerate}
\item $U\setminus E$ is dense in $U$.
\item Every $f\in H^{\infty}(U\setminus E)$ has a unique norm preserving analytic extension to $U$.
\item $\gamma(E)=0$.
\end{enumerate}
\end{lemma}

\begin{proof}
(1) Suppose that $U\setminus E$ is not dense in $U$. Then, there exists a point $a\in U$ 
such that 
\begin{equation*}
\inf_{z\in U\setminus E}|z-a|>0,
\end{equation*}
and hence $f(z):=\frac{1}{z-a}$ is in $H^{\infty}(U\setminus E)$. However, the identity theorem tells that its analytic extension to $U$ must be $\frac{1}{z-a}$ on $U\setminus \{a\}$, a contradiction.

(2) This immediately follows from (1).

(3) We must show that $\gamma(K)=0$ for every compact subset $K\subset E$. Since $H^{\infty}(U\setminus K)\subset H^{\infty}(U\setminus E)$ (by the restriction map), for any $f\in H^{\infty}(U\setminus K)$ there exists an $F\in H^{\infty}(U)$ such that $f(z)=F(z)$ for all $z\in U\setminus E$ ($\subset U\setminus K$). Hence, we have $f=F$ on $U\setminus K$ by (1).
Thus, $\gamma(K)=0$ follows from \cite[Theorem 1.10(ii)$\Rightarrow$(iv)]{Tol}. 
\end{proof}


We also need the next lemma to prove the main theorem of this section.

\begin{lemma}\label{connected}
If $H_k$ is of Hardy type on a reduced complex space $X$, then $X$ must be connected.
\end{lemma}
\begin{proof}
 Suppose that $X$ can be divided into two disjoint non-empty open sets $X_0$ and $X_1$. We may assume that $H_k$ is normalized at $x_0$ and $x_0\in X_0$. We define a holomorphic function $\phi:X\rightarrow\C$ by 
 \begin{align*}
    \phi(x)&=
    \begin{cases}
        1&(x\in X_0),\\
        0&(x\in X_1).
    \end{cases}
 \end{align*}
 Since $H_k$ is of Hardy type, this is a contractive multiplier of $H_k$. Therefore, the maximum modulus principle for multipliers \cite[Lemma 2.2]{AHMR} shows that $\phi$ must be a constant function. This is a contradiction.
\end{proof} 

\begin{example}
\upshape For any $0<s<t<1$, we define an open set $C(s,t)\subset\D$ by
\begin{equation*}
C(s,t)=\{z\in\C\;|\;0\le|z|<s\}\cup\{z\in\C\;|\;t<|z|<1\}.
\end{equation*}
Let $H_{k_{s,t}}$ be a complete Nevanlinna--Pick space of holomorphic functions on $C(s,t)$ given by reproducing kernel 
\begin{equation*}
k_{s,t}(z,w)=\frac{1}{1-z\overline{w}}\;\;\;\;(z,w\in C(s,t)).
\end{equation*}
Then, $H_{k_{s,t}}$ never becomesz of Hardy type because $C(s,t)$ is not connected. However, $H_{k_{s,t}}$ is of weak Hardy because the maximum modulus principle implies that $C(s,t)$ is dominating.
\end{example}

Here is the main result of this section.

\begin{theorem}\label{theorem2}
Let $X$ be a reduced complex space with arbitrary $x_0\in X$. Then, there exists an RKHS $H_k$ of Hardy type on $X$ normalized at $x_0$ if and only if there exist a relatively closed set $E\subset\D$ and a 
biholomorphic map $j$ from X onto $\D\setminus E$ such that 
\begin{itemize}
\item[(1)] $\D\setminus E$ is connected, 
\item[(2)] $\gamma(E)=0$,
\item[(3)] $j$ sends $x_0$ to the origin $0 \in \mathbb{D}$.
\end{itemize}
Moreover, in this case, such an RKHS $H_k$ is unique and its reproducing kernel is given by
\begin{equation*}
     k(x,y)=\frac{1}{1-j(x)\overline{j(y)}}\;\;\;\;(x,y\in X).
\end{equation*}
Hence, 
\begin{equation*}
    U: H^2\rightarrow H_k, \;\;\;\; f\mapsto f\circ j
\end{equation*}
is a unitary operator. 
\end{theorem}

We remark that item $(2)$ implies item $(1)$. In fact, if $\D\setminus E$ is not connected, then we have a holomorphic function that takes $1$ on a connected component and $0$ on the others. Then, $\gamma(E)=0$ ensures, by Lemma \ref{lemma:capacity}, that the function admits an analytic extension to the whole $\D$, a contradiction.

\begin{proof}
The uniqueness of $H_k$ follows from \cite[Corollary 3.2]{Har2017}.

Let $H_k$ be an RKHS of Hardy type on $X$ normalized at $x_0\in X$. Theorem \ref{theorem1} shows that there exists an injection $j:X\rightarrow\D$ such that $j\in\M$, $j(x_0)=0$ and
\begin{equation*}
 k(x,y)=\frac{1}{1-j(x)\overline{j(y)}}\;\;\;\;(x,y\in X).
\end{equation*}
Since $H_k$ is of Hardy type, $j$ is holomorphic on $X$. Set $E:=\D\setminus j(X)$.
By Lemma \ref{connected}, $X$ is connected and so is $j(X)$ too. 

We will prove that $X$ is a pure $1$-dimensional reduced complex space.
Since $X$ is connected but not a singleton, we have
$\dim_x{X}\ge 1$
for all $x\in X$. Moreover, the injectivity of the holomorphic function $j$ shows that 
\begin{equation*}
\{\mathfrak{z}\in X\;|\;j(\mathfrak{z})-j(x)=0\}=\{x\}
\end{equation*}
for each $x\in X$. Hence, every $x\in X$ satisfies
$\dim_x{X}\le 1$.
Therefore, $X$ must be pure $1$-dimensional.
Therefore, \cite[Theorem in page 166]{GR} tells us that $j$ is open and a biholomorphic map from $X$ onto $\D\setminus E$. 

It remains to prove $\gamma(E)=0$. For any $\psi\in H^{\infty}(\D\setminus E)$, we define a bounded holomorphic function on $X$ by $\psi\circ j$. Since $H_k$ is of Hardy type, $\psi\circ j$ falls in $\M$. Thus Lemma \ref{lemma1} yields $\Phi\in H^{\infty}(\D)$ such that $\Phi(z)=\psi(z)$ for all $z\in \D\setminus E$. By Lemma \ref{lemma:capacity}(3), we get $\gamma(E)=0$. 

To prove the converse direction, we set 
\begin{equation*}
k(x,y):=\frac{1}{1-j(x)\overline{j(y)}}\;\;\;\;(x,y\in X). 
\end{equation*}
Hence $H_k$ is an irreducible complete Nevanlinna--Pick space of holomorphic functions on $X$ normalized at $x_0$. Since $1=k(\cdot, x_0)\in H_k$ and $\|\phi\|_{\infty}\le\|\phi\|_{\M}$ for all $\phi\in \M$, we have $\M\subset H^{\infty}(X)$. For any $\phi\in H^{\infty}(X)$ with $\|\phi\|_{\infty}=1$ we define a $\psi\in H^{\infty}(\D\setminus E)$ by $\psi:=\phi\circ j^{-1}$. As $\gamma(E)=0$, Lemma \ref{capacity} and Lemma \ref{lemma:capacity}(2) yield an analytic extension $\Psi\in H^{\infty}(\D)$ of $\psi$ with 
\begin{equation*}
\|\Psi\|_{\infty}=\|\psi\|_{\infty}=1.
\end{equation*}
Hence, $\Psi$ is a contractive multiplier of $H^2$ and we can see that $\phi$ is a contractive multiplier of $H_k$. In fact, we have
\begin{equation*}
    (1-\phi(x)\overline{\phi(y)})k(x,y)=\frac{1-\psi(j(x))\overline{\psi(j(y))}}{1-j(x)\overline{j(y)}}=\frac{1-\Psi(j(x))\overline{\Psi(j(y))}}{1-j(x)\overline{j(y)}}.
\end{equation*}
Therefore, $\M$ is isometrically equal to $H^{\infty}(X)$.
\end{proof}

If $d\ge 2$ and $X=\B_d$ or $X=\D^d$, it is known that there is no RKHS of Hardy type on $X$ (see \cite[Proposition 8.83]{AMp} and \cite[Corollary 3.3]{Har2017}). 
As an application of Theorem \ref{theorem2}, we can generalize this fact. 

\begin{cor}\label{corollary1}
If $X$ is either a reduced complex space with a singular point or a complex manifold with $\dim{X}\ge 2$, then there is no RKHS of Hardy type on $X$.
\end{cor}

\begin{example}
\upshape The origin $(0,0)$ is a singular point of Neil's parabola 
\begin{equation*}
\{(z,w)\in\C^2\;|\; z^2-w^3=0\}.
\end{equation*}
Therefore, there is no RKHS of Hardy type on Neil's parabola.
\end{example}


\section{Characterization of $X$ biholomorphic to the open unit disk $\D$}

We have already seen that there exists an RKHS of Hardy type on the open unit disk minus a singleton.
Therefore, the existence of such a space does not, in general, imply that $X$ is biholomorphic to the open unit disk $\D$.  However, under the additional hypothesis that $X$ is complete with respect to the Carath\'{e}odory distance on $X$, we can prove that $X$ must be biholomorphic to $\D$. 
We define the M$\ddot{\mathrm{o}}$bius pseudo-distance $d_X$ on a reduced complex space $X$ by
\begin{equation*}
d_X(x,y):=\sup_{f\in H(X,\D)}d_{\D}(f(x), f(y))\;\;\;\;(x,y\in X),
\end{equation*}
where $\displaystyle d_{\D}(z,w):=\left|\frac{z-w}{1-\overline{w}z}\right|$ and $H(X,\D)$ is the set of holomorphic functions from $X$ into $\D$. The Carath\'{e}odory pseudo-distance on $X$, denoted by $c_X$, is given by 
\begin{equation*}
c_X(x,y):=\tanh^{-1}(d_X(x,y)).
\end{equation*}

\begin{definition}
We say that a reduced complex space $X$ is {\bf complete with respect to the Carath\'{e}odory distance} if $c_X$ is a distance and $X$ is Cauchy complete with respect to $c_X$. Equivalently, $d_X$ is a distance and $X$ is Cauchy complete with respect to $d_X$
\end{definition}


If $H_k=H^2$, then it is obvious that $d_k=d_{\D}$. The following lemma is a generalization of this observation:
\begin{lemma}\label{lemma4-1}
Let $H_k$ be an irreducible RKHS of holomorphic functions on a reduced complex space $X$. If $H_k$ satisfies the two-point Nevanlinna--Pick property and $\M=H^{\infty}(X)$ isometrically, then 
\begin{equation*}
d_X(x,y)=d_k(x,y)
\end{equation*}
 holds for every $x,y\in X$.
\end{lemma}
\begin{proof}
Since $\M$ is isometrically equal to $H^{\infty}(X)$, \cite[Theorem 6.22]{ARSW2019} implies that $d_X(x,y)\le d_k(x,y)$ for all $x,y\in X$. For an arbitrary point $x_0\in X$ we set a normalized reproducing kernel $k_{x_0}$ by equation (\ref{equation:normalization}). Then, it is easy to check that 
\begin{equation*}
d_k(x,x_0)=d_{k_0}(x,x_0)
\end{equation*} 
holds; see Remark \ref{remark:invariant}.
By Lemma \ref{lemma2}, we have 
\begin{equation*}
d_k(x,x_0)=\sup\{\mathrm{Re}\phi(x)\;|\;\|\phi\|_{\M}\le 1 \;\mathrm{and}\; \phi(x_0)=0\}
\end{equation*}
for all $x\in X$. 
Moreover, there is a multiplier $\phi_x$ that achieves the supremum if $x\ne x_0$. 
The maximum modulus principle for multipliers \cite[Lemma 2.2]{AHMR} shows $\phi_x\in H(X,\D)$. Moreover, we have
\begin{equation*}
d_k(x,x_0)=\mathrm{Re}\phi_x(x)\le|\phi_x(x)|=d_{\D}(\phi_x(x), \phi_x(x_0))\le d_X(x,x_0).
\end{equation*}
Since $x$ and $x_0$ are arbitrary, we are done.
\end{proof}

If $H_k$ is an RKHS of Hardy type on a reduced complex space $X$,  then the above lemma and \cite[Lemma 9.9]{AMp} imply that $d_X$ is a distance. Hence, so is $c_X$.


Here is the main result of this section.

\begin{theorem}\label{prop1}
    Let $H_k$ be an RKHS of Hardy type on a reduced complex space $X$.
    If $X$ is complete with respect to the Carath\'{e}odory distance, then $X$ must be biholomorphic to $\D$.
\end{theorem}
\begin{proof}
By theorem \ref{theorem2}, there exist a relatively closed set $E\subset\D$ and a biholomorphic map $j$ from $X$ onto $\D\setminus E$ such that $\gamma(E)=0$ and 
\begin{equation*}
    k(x,y)=\frac{1}{1-j(x)\overline{j(y)}}\;\;\;\;(x,y\in X).
\end{equation*}
We will prove that $E=\emptyset$. 

For any $z\in \D$ we can choose a sequence $\{x_n\}_{n=1}^{\infty}$ in $X$ such that $\displaystyle\lim_{n\to\infty}j(x_n)=z$ (see Lemma \ref{lemma:capacity}(1)). Note that Lemma \ref{lemma4-1} immediately implies that
\begin{equation*}
d_X(x,y)=\sqrt{1-\frac{(1-|j(x)|^2)(1-|j(y)|^2)}{|1-\overline{j(y)}j(x)|^2}}=d_{\D}(j(x), j(y)).
\end{equation*}
Therefore, we have $\displaystyle \lim_{n, m\to \infty}d_X(x_n, x_m)=\lim_{n,m\to\infty}d_{\D}(j(x_n), j(x_m))=0$. Since $X$ is complete with respect to the Carath\'{e}odory distance, there exists a point $x\in X$ such that $\displaystyle \lim_{n\to\infty}d_X(x_n, x)=0$. Hence, we have $\displaystyle z=\lim_{n\to\infty}j(x_n)=j(x)$ and $E=\emptyset$.
\end{proof}


\begin{cor}\label{corollary2}
    Let $X$ be a reduced complex space. If $X$ is complete with respect to the Carath\'{e}odory distance but not biholomorphic to $\D$, then there is no RKHS of Hardy type on $X$.
\end{cor}

\begin{example}\label{example:annulus}
\upshape For $r\in\mathbb{R}$ with $0<r<1$, the annulus 
\begin{equation*}
A(r):=\{z\in \C\;|\;r<|z|<1\}
\end{equation*}
is complete with respect to the Carath\'{e}odory distance but not biholomorphic to the open unit disk $\D$. Therefore, there is no RKHS of Hardy type on $A(r)$.
\end{example}

\begin{remark}\label{remark:last}
\upshape (1) Arcozzi, Rochberg and Sawyer \cite[Corollary 13]{ARS} proved that, if $X$ is a domain in $\C$ with boundary consisting of finitely many smooth curves, then there exists an irreducible complete Nevanlinna--Pick space $H_k$ whose multiplier and supremum norms are equivalent.
Therefore, there exists such a space on the annulus $A(r)$ in contrast to Example \ref{example:annulus}.

(2) Let $X$ be a bounded domain in $\C$. Selby \cite{Sel} gave some necessary and sufficient conditions for $X$ to be complete with respect to the Carath\'{e}odory distance. Since we are now considering an RKHS of Hardy type on $X$, the conditions in \cite{Sel} are equivalent to the conditions (i')-(iii') in \cite[Proposition 4.2]{AHMR}. In fact, these results share the condition that the union of all fibers corresponding to points of $X$ forms a Gleason part in the maximal ideal space of $H^{\infty}(X)=\M$. 
\end{remark}


\section*{Acknowledgment}
The author gratefully acknowledges Professor John Edward McCarthy, who was the author's host in JSPS Overseas Challenge Program for Young Researchers, for fruitful discussions. 
The author also acknowledges Dr.\,George Tsikalas for his useful comments, and for letting the author know the work \cite{ARS} (concerning Remark \ref{remark:last} (1)).
The author is grateful to his supervisor Professor Yoshimichi Ueda for his comments on the draft of this paper. This work was supported by JSPS Research Fellowship for Young Scientists (KAKENHI Grant Number JP 23KJ1070).


\end{document}